\documentclass[12pt]{amsart}

\usepackage{amssymb,latexsym}
\usepackage{amsmath,amsfonts,amsthm,amscd,amsxtra,enumerate}
\usepackage{mathtools}
\usepackage[all]{xy}
\usepackage{hyperref}
\usepackage{amsmath,latexsym,amssymb, enumerate, amsthm, epsfig }
\usepackage{bookmark}
\usepackage[margin=1in]{geometry}
\usepackage{epsfig}
\usepackage{graphicx}
\usepackage{enumerate}

\newcommand{\noi}{\noindent}
\newcommand{\be}{\begin{enumerate}}
	\newcommand{\ee}{\end{enumerate}}

\newtheorem{con}{Conjecture}[section]
\newtheorem{defn}[con]{Definition}

\newtheorem{thm}[con]{Theorem}
\newtheorem*{thm*}{Theorem}
\newtheorem{lemma}[con]{Lemma}
\newtheorem{remark}[con]{Remark}
\newtheorem{corollary}[con]{Corollary}
\newtheorem{proposition}[con]{Proposition}

\newtheorem*{con*}{Conjecture}
\AtBeginDocument{}
\begin{document}
\title{ Some results on the Ryser design conjecture-III}
\author[Tushar D. Parulekar]{Tushar D. Parulekar}
\address{Department of Mathematics, I.I.T. Bombay, Powai, Mumbai 400076.}
\email{tushar.p@math.iitb.ac.in}
\author[Sharad S. Sane]{Sharad S. Sane}
\address{Chennai Mathematical Institute, SIPCOT IT Park, Siruseri, Chennai 603103}
\email{ssane@cmi.ac.in}
\date{\today}

\subjclass[2010]{05B05;    51E05;     62K10}

\keywords{Ryser design}
\begin{abstract}
	A Ryser design $\mathcal{D}$ on $v$ points is a collection of $v$ proper subsets (called blocks) of a point-set with $v$ points such that every two blocks intersect each other in $\lambda$ points (and $\lambda < v$ is a fixed number) and there are at least two block sizes. A design $\mathcal{D}$ is called a symmetric design, if every point of $\mathcal{D}$ has the same replication number (or equivalently, all the blocks have the same size) and every two blocks intersect each other in $\lambda$ points. The only known construction of a Ryser design is via block complementation of a symmetric design. Such a Ryser design is called a Ryser design of  Type-1. This is the ground for the Ryser-Woodall conjecture: ``every Ryser design is of Type-1". This long standing conjecture has been shown to be valid in many situations.
	Let $\mathcal{D}$ denote a Ryser design of order $v$, index $\lambda$ and replication numbers $r_1,r_2$. Let $e_i$ denote the number of points of $\mathcal{D}$ with replication number $r_i$ (with $i = 1, 2$). Call a block $A$ of $\mathcal{D}$ small (respectively large) if $|A| < 2\lambda$ (respectively $|A| > 2\lambda$) and average if $|A|=2\lambda$. Let $D$ denote the integer $e_1 - r_2$ and let $\rho> 1$ denote the rational number $\dfrac{r_1-1}{r_2-1}$. Main results of the present article are the following: An equivalence relation on the set of Ryser designs is established. Some observations on the block complementation procedure of Ryser-Woodall are made. It is shown that a Ryser design with two block sizes one of which is an average block size is of Type-1. It is also shown that, under the assumption that large and small blocks do not coexist in any Ryser design equivalent to a given Ryser design, the given Ryser design must be of Type-1.
\end{abstract}
\maketitle
\section{Introduction}
\noindent
A design is a pair $(X,L)$, where $X$ is a finite set of points and $L\subseteq P(X)$, where $P(X)$ is the power set of $X$. The elements of $X$ are called its points and the members of $L$ are called the blocks. Most of the definitions, formulas and proofs of standard results used here can be found in \cite{Ion2}.
%\newpage
\begin{defn}{\rm
		A design $\mathcal{D}=(X,L)$ is said to be a symmetric $(v,k,\lambda)$ design  if
		\begin{enumerate}[{\rm 1.}]
			\item $ |X|=|L|=v$,
			\item $ |B_1\cap B_2|=\lambda \geq 1$ for all blocks $B_1$ and $B_2$ of $\mathcal{D},~~ B_1\neq B_2$,
			\item $ |B|=k>\lambda$ for all blocks $B$ of $\mathcal{D}$.
		\end{enumerate}
}\end{defn}
\newpage
\begin{defn}{\rm
		A design $\mathcal{D}=(X,L)$ is said to be a Ryser design of order $v$ and index $\lambda$ if it satisfies the following:
		\begin{enumerate}[{\rm 1.}]
			\item $|X|=|L|=v$,
			\item $ |B_1\cap B_2|=\lambda$ for all blocks $B_1$ and $B_2$ of $\mathcal{D}, B_1\neq B_2$,
			\item $ |B|>\lambda$ for all blocks $B$ of $\mathcal{D}$,
			\item there exist blocks $B_1$ and $B_2$ of $\mathcal{D}$ with $|B_1|\neq|B_2|$.
		\end{enumerate}
}\end{defn}
\noindent
Here condition 4 distinguishes a Ryser design from a symmetric design, and condition 3 disallows repeated blocks and also any block being contained in another block.\\% The term $\lambda$ is called the index of the Ryser design.\\
Woodall  \cite{wodl} introduced a type of combinatorial object which is a combinatorial dual of a Ryser design. All the known examples of Ryser designs can be described by the following construction %given by Ryser 
which is also known as the Ryser-Woodall complementation.\\
Let $\mathcal{D}=(X,\mathcal{A})$ be a symmetric $(v,k,k-\lambda)$ design with $k\neq 2\lambda$. Let $A$ be a fixed block of $\mathcal{D}$.
Form the collection $\,\mathcal{B}=\{A\}\bigcup \{A\triangle B: B\in \mathcal{A}, B\neq A\}$, where $A\triangle B$ denotes the usual symmetric difference of $A$ and $B$. Then $\overline{\mathcal{D}}=(X,\mathcal{B})$ is a Ryser design of order $v$ and index $\lambda$ obtained from $\mathcal{D}$ by block complementation with respect to the block $A$. We denote $\overline{\mathcal{D}}$ by $\mathcal{D}*A$. %We can take points of $\mathcal{D}*A$ to be the same as those of $\mathcal{D}$.\\
Then $A$ is also a block of $\mathcal{D}*A$ and the original design $\mathcal{D}$ can be obtained by complementing $\mathcal{D}*A$ with respect to the block $A$. 
If $\mathcal{D}$ is a symmetric $(v,k,\lambda^{'})$ design, then the design obtained by complementing $\mathcal{D}$ with respect to some block is a Ryser design of order $v$ with index $\lambda=k-\lambda^{'}$. A Ryser design obtained in this way is said to be of \textbf{Type-1}.\\
Define a Ryser design to be of \textbf{Type-2} if it is not of Type-1. We now state
\begin{center}
	\textbf{The Ryser Design Conjecture \cite{Ion2}: Every Ryser design is of Type-1.}
\end{center}
\noindent
In a significant paper Singhi and Shrikhande \cite{Sin} proved the conjecture when the index $\lambda$ is a prime. In  \cite{Ser2} Seress showed the truthfulness of the conjecture for $\lambda=2p$, where $p$ is a prime. In \cite{Ion1} Ionin and Shrikhande  developed a new approach to the Ryser design conjecture that led to new results for certain parameter values. They also gave an alternative proof of the celebrated non-uniform Fisher Inequality. Ionin and Shrikhande went on to explore the validity of the Ryser design conjecture from a different perspective. Their results prove the conjecture for certain values of $v$ rather than for $\lambda$. Both Ryser and Woodall independently proved the following result:
\begin{thm}[{\cite[Theorem 14.1.2]{Ion2}}{ Ryser Woodall Theorem}]\label{thm:RyserWoodall}
	If $\mathcal{D}$ is a Ryser design of order $v$, then there exist integers $r_1$ and $r_2$, $r_1\neq r_2$ such that $r_1+r_2=v+1$ and any point occurs either in $r_1$ blocks or in $r_2$ blocks.
\end{thm}
\noindent
  %Let $g$ denote the gcd between $r_1 - 1$ and $r_2 - 1$ and let $c$ and $d$ respectively denote the integers $\dfrac{r_1 - 1}{g}$ and $\dfrac{r_2 - 1}{g}$. Let $\rho=\dfrac{r_1-1}{r_2-1}=\dfrac{c}{d}$. 
  The point-set is partitioned into subsets $E_1$ and $E_2$, where $E_i$ is the set of points with replication number $r_i$ and let $e_i=|E_i|$ for $ i = 1, 2$. Then $e_1,e_2 > 0$ and $e_1 + e_2 = v$.
For a block $A$, let $\tau_i(A)$ denote $|E_i\cap A|$, the number of points of block $A$ with replication number $r_i$ for $i=1,2$.
Then $|A|=\tau_1(A)+\tau_2(A)$. \textit{We say a block $A$ is large, average or small if $|A|$ is greater than $2\lambda$, equal to $2\lambda$ or less than $2\lambda$ respectively.  A block which is not average is called a non average block.}  \\
The Ryser-Woodall complementation (or block complementation) of a Ryser design $\mathcal{D}$ of index $\lambda$ with respect to some block $\,A \in \mathcal{D}\,$ is either a symmetric design or a Ryser design of index $(|A|-\lambda)$. If $\mathcal{D}*A$ is the new Ryser design of index $(|A|-\lambda)$ obtained by Ryser-Woodall complementation of a Ryser design $\mathcal{D}$ with respect to the block $A$, we denote the new parameters of $\mathcal{\mathcal{D}*A}$ by $\lambda(\mathcal{D}*A), e_1(\mathcal{D}*A)$ etc.\\
Let $\mathcal{D}_r(X)$ denote the set of all incidence structures $\mathcal{D}=(X,\mathcal{B})$ where $\mathcal{B}$ is a set of subsets of $X$ and $\mathcal{D}$ is a Ryser design with replication numbers $r_1 \text{ and } r_2=v+1-r_1$; or a symmetric design with block size $r_1 \text{ or } r_2$.
%We now quote some useful results from Ionin and Shrikhande \cite{Ion2}.
\begin{proposition}[{\cite[Proposition 14.1.7]{Ion2}}] \label{prop:complement-properties}
	Let $\mathcal{D}\in \mathcal{D}_r(X)$ and let $A,B$ be blocks of 
	$\mathcal{D}$. Then $\mathcal{D}* A\in \mathcal{D}_r(X)$ and the following 
	conditions hold: 
	\begin{enumerate}[{\rm (i)}]
		\item $(\mathcal{D}*A)*A=\mathcal{D}$;
		%if $B$ is a block o Df $b\neqA a$ then B\trangle A is a block of D*A and D*A*A\trangle B =D*B 
		\item $A\triangle B$ is a block of $\mathcal{D}*A$ and $(\mathcal{D}*A 
		)*(A\triangle B)=\mathcal{D}*(B)$;
		\item $r_1(\mathcal{D}*A)=r_1(\mathcal{D})$;
		\item $\lambda(\mathcal{D}*A)=|A|-\lambda(\mathcal{D})$; 
		\item $E_1(\mathcal{D}*A)=E_1(\mathcal{D})\triangle A$;
		\item 
		$e_1(\mathcal{D}*A)=e_1(\mathcal{D})-\tau_1(A)(\mathcal{D})+
		\tau_2(A)(\mathcal{D}
		)$;
		\item $\mathcal{D}*A$ is a symmetric design if and only if 
		$A=E_1(\mathcal{D})\text{ or } A=E_2(\mathcal{D})$. 
	\end{enumerate}
\end{proposition}
\begin{remark}\label{remark:1}
	Since $|A\triangle B|=|A|+|B|-2|A\cap B|$,  observe that if a design is of Type-1 then it has all average blocks except for one, and hence a Type-2 Ryser design must have at least two non average blocks.   
\end{remark}
%\begin{thm}[{\cite[Theorem 14.1.17]{Ion2}}] 
%	\label{thm:rho<lambda}
%	For any Ryser design with block intersection $\lambda>1$ and replication numbers $r_1 \text{ and } r_2$, 
%	$\quad\displaystyle{\frac{\lambda}{\lambda-1}\leq \rho \leq \lambda}\text{ and } \rho \notin (\lambda-1,\lambda)$, where $\rho=\dfrac{r_1-1}{r_2-1}$.
%\end{thm}    
%
%\begin{thm}[{\cite[Theorem 1.7]{Ser1}}]\label{thm:e1e2=lambda(v-1)}
%	A Ryser design is of Type-1 if and only if $\,e_1e_2=\lambda(v-1)$.
%\end{thm}
%
%\begin{thm}[{\cite[Theorem 14.1.20]{Ion2}}]\label{thm:a<lambda}
%	If $r_1, r_2$ are the replication numbers of a Ryser design of index $\lambda>2$, then $r_1-r_2\leq (\lambda-1)g, \text{ where } g=\gcd(r_1-1,r_2-1)$.    
%\end{thm}

\noindent
%Following Seress {\cite{Ser2}} we define $D = e_1 - r_2$ and 
Following Singhi and Shrikhande {\cite{Sin}} we define $\rho=\dfrac{r_1-1}{r_2-1}=\dfrac{c}{d}$, where $\gcd(c,d)=1.$ 
Let $g=\gcd(r_1-1,r_2-1). \text{ 
	Then } r_1+r_2=v+1 \text{ implies } g \text{ divides } (v-1)$,  
$~~r_1-1=cg,~ r_2-1=dg \text{ and } v-1=(c+d)g$.
We also write $a$ to denote $c - d$ and observe that any two of $c, d$ and $a$ are coprime. 
%By Theorem \ref{thm:a<lambda}$~~$ we have ;
%\begin{equation}\label{a_less_than_lambda}
%a<\lambda
%\end{equation}
We use the following equations which can be found in \cite{Sin} and \cite{Ion1}.\\
In a Ryser design with block sizes $ k_1,k_2,.....k_v$  
\begin{equation}\label{sum_of_all}
\sum_{m=1}^{v}\frac{1}{k_m-\lambda}=\frac{(\rho +1)^2}{\rho}-\frac{1}{\lambda}
\end{equation}
\begin{align}
&e_1r_1(r_1 - 1) + e_2r_2(r_2 - 1) = \lambda v(v - 1) \\%\label{e1r1r1-1+e2r2r2-1}\\
&(\rho-1)e_1=\lambda(\rho + 1)-r_2 \label{e1form}\\
%&e_1=\lambda + \frac{\lambda+D}{\rho} \label{e1Dform}\\
&(\rho-1)e_2=\rho r_1-\lambda(\rho + 1) \label{e2form}
%&e_2= \lambda +[\lambda -(D+1)]\rho \label{e2Dform}
\end{align}
%\noindent
%From equations (\ref{e2form}) and (\ref{e1form}), respectively:
%\begin{align}
%&r_1 = 2\lambda + \left(\frac{a}{c}\right)(e_2-\lambda) \label{r1form}\\
%&r_2=2\lambda-\left(\frac{a}{d}\right)(e_1-\lambda) \label{r2form}
%\end{align}
\noindent
%In a Ryser design a block $A$ with $|A| = 2\lambda$ is called an average block. We also define $A$ to be small (respectively large) if $|A| < 2\lambda$ (respectively if $|A| > 2\lambda$).
%\begin{equation}\label{e1rhopluse2byrho_this_ch}
%1 +\rho e_1+\dfrac{e_2}{\rho}= \lambda\dfrac{(\rho +1)^2}{\rho}
%\end{equation}
\noi
Let $A$ be a block of a Ryser design $\mathcal{D}$
 with $|A|=\tau_1(A)+\tau_2(A)$, a simple two way counting gives,
\begin{equation*}\label{tau1r1-1+tau2r2-1}
(r_1-1)\tau_1(A)+(r_2-1)\tau_2(A)=\lambda(v-1)
\end{equation*}
which gives
%\begin{equation}\label{tau1rhotau2form}
%\rho\tau_1(A)+\tau_2(A)=\lambda (\rho + 1).
%\end{equation}
%\noindent
%this gives us,
%Dividing equation (\ref{tau1r1-1+tau2r2-1}) through by $g$, the common g.c.d. of  $r_1 -1, r_2 - 1$ and $v - 1$ yields: $ c\tau_1(A) + d\tau_2(A) = \lambda(c + d) $ and hence 
%$ c( \tau_1(A) - \lambda)  + d( \tau_2(A) - \lambda) = 0. $
%Using the coprimality of $c$ and $d$ then shows that $c$ divides $\tau_2(A) - \lambda$ while $d$ divides  $\tau_1(A) - \lambda$ and writing the ratios to be $t$ and $s$ respectively, it is clear that $s = - t$ and hence $\tau_1(A) - \lambda = -td$ and 
%$\tau_2(A) - \lambda = tc$ for some integer $t$. 
%That is;
\begin{align*}
&\tau_1(A)=\lambda-td\\%\label{tau1form}\\
&\tau_2(A)=\lambda+tc\\%\label{tau2form}\\
&|A|=2\lambda+ta%\label{sizeofAform}
\end{align*} for some integer $t$. These findings are summed in the following
lemma.
\begin{lemma}\label{lemma:blocksize}
	Let $A$ be a block of a Ryser design. Then the size of $A$ has the form  $|A|=2\lambda+ta$, where $t$ is an integer. The block $A$ is large, average or small depending on whether $t>0,  t=0$ or $t<0$ respectively. Hence $\tau_1(A)=\tau_2(A)=\lambda$ if $A$ is an average block,  $\tau_1(A)>\lambda>\tau_2(A)$ if $A$ is a small block and  $\tau_2(A)>\lambda>\tau_1(A)$ if $A$ is a large block.
\end{lemma}
\noindent
In a Ryser design $\mathcal{D}$ with blocks sizes $|A_i|=k_i \text{ for } i=1,2,\ldots,v $ the column sum of the incidence matrix is equal to the row sum of the incidence matrix which implies $\sum k_i= e_1r_1+e_2r_2$.
Hence from equation (\ref{e1form}) and (\ref{e2form}), we get 
\begin{equation}\label{e1r1e2r2form}
e_1r_1+e_2r_2=\lambda(v-1)+r_1r_2
\end{equation}
%In equation (\ref{r1form}), let $x=\dfrac{e_2-\lambda}{c}$ and in equation (\ref{r2form}), let $y=\dfrac{e_1-\lambda}{d}$. 
%Then, $r_1 = 2\lambda +xa$ and $r_2 = 2\lambda -ya$.  Since $c$ and $a$ are co-prime, it follows at once that $c$ divides $e_2 - \lambda$ and hence $x$ is an integer. By equation (\ref{e2Dform}) we get;
%\begin{equation}\label{xform}
%x=\dfrac{e_2-\lambda}{c}=\dfrac{\lambda -(D+1)}{d}
%\end{equation}
%%Which gives us
%%\begin{equation}\label{e2lambdaxcform}
%%e_2=\lambda+xc
%%\end{equation}
%The assertion that $y$ is an integer follows similarly.  By equation (\ref{e1Dform}) we get;
%\begin{equation}\label{yform}
%y=\dfrac{e_1-\lambda}{d}=\dfrac{\lambda +D}{c}
%\end{equation}
%%Which gives us
%%\begin{equation}\label{e1lambdaydform}
%%e_1=\lambda+yd
%%\end{equation}
%\noindent
%%$x=\dfrac{e_2-\lambda}{c} $ 
%%Similarly,
%%$y=\dfrac{e_1-\lambda}{d}$ 

\noindent
In this article, a binary set operation $\triangle$ is defined that gives an equivalence relation on the set of Ryser designs of order $v$. Some
observations on the block complementation procedure of Ryser-Woodall
are made. It is shown that a Ryser design of order $v$ and index $\lambda$ with two block sizes and with one block size $2\lambda$ is of Type-1.
It is also shown that, under the assumption that large and small blocks do not coexist in any Ryser design equivalent to a given Ryser design, the given Ryser design must be of Type-1.

\section{\bf{Equivalence classes of Ryser designs}}
\noi The binary set operation $A\triangle B=(A\cap B^c)\cup(A^c\cap B)$ is well known and $A\triangle B$ is the set of all the elements that are in precisely one of the sets $A$ and $B$. The following (Boolean algebraic) lemma is also well known and hence the elementary proof is omitted. 
%\vspace{.5cm}
%\noi {\bf Lemma 1:} 
\begin{lemma}The binary set operation $A\triangle B$ has the following properties:
\be
\item[(a)] $\bigtriangleup$ is commutative. Further, $A \bigtriangleup B = A$ if and only if 
$B = \emptyset$. Also $A \bigtriangleup A = \emptyset$. 
\item [(b)] $\bigtriangleup$ is associative. In fact the set 
$A_1 \bigtriangleup A_2  \bigtriangleup \cdots  \bigtriangleup A_n$
precisely consists of those elements that belong to an odd number of $A_i$s. 
\ee
\end{lemma}

\noi
Let $V$ be a $v$-set ({\it which is now fixed for the entire discussion to follow}). Let $\Omega$ be a family of $v$ distinct non-empty subsets of $V$. {\it Members of $\Omega$ are called its blocks}. Let $A \in \Omega$. Define 
a function (block complementation w.r.t. $A$) $f_A$ on $\Omega$ as follows: $f_A(A) = A$ and for all $B \neq A$ define $f_A(B) = A \bigtriangleup B$. Then ${\Omega}' = f_A({\Omega})$ is also a family of $v$ subsets of $V$. We also emphasize that $f_A$ is not defined on $\Omega$ if $A \notin \Omega$. The following lemma is then obvious. 
%
%\vspace{.5cm}
%
%\noi {\bf Lemma 2:} 
\begin{lemma}
With everything as above, let ${\Omega}' = f_A({\Omega})$. Then:
\be
\item[(a)] $A \in {\Omega}'$. 
\item[(b)] $ {\Omega}'$ is also a collection of $v$ distinct non-empty subsets of  $\Omega$. 
\item[(c)] $f_A({\Omega}') = \Omega$. 
\ee
\end{lemma}
%\vspace{.5cm}

%\noi {\bf Lemma 3 (The Basic Lemma):} 
\begin{lemma}\label{lemma:map}
Let $f_A({\Omega}) = {\Omega}'$ and let $f_{A \bigtriangleup B}({\Omega}') = {\Omega}''$ where $A$ and $B$ are distinct subsets in 
$\Omega$.  Then ${\Omega}'' = f_B(\Omega)$. 
\end{lemma}
\begin{proof}
 Let ${\Omega}^* = f_B(\Omega)$. Then besides $B$, ${\Omega}^*$ contains all the sets of the form $B \bigtriangleup C$ where $C \in \Omega$ and $C \neq B$. Let $C \in \Omega$ and $C \neq A, B$. Then 
$$f_{A \bigtriangleup B}f_A(C) = f_{A \bigtriangleup B}(A \bigtriangleup C) = (A \bigtriangleup B) \bigtriangleup (A \bigtriangleup C) = (B \bigtriangleup C)$$
Further, $f_{A \bigtriangleup B}f_A(A) = f_{A \bigtriangleup B}(A) = (A \bigtriangleup B) \bigtriangleup A = B \text{ and } f_{A \bigtriangleup B}f_A(B) = f_{A \bigtriangleup B}(A \bigtriangleup B) = A \bigtriangleup B$. This shows that ${\Omega}^* = {\Omega}''$. 
\end{proof}
%\vspace{.5cm}

\noi {\it We also define a generic universal function on a family $\Omega$ of subsets of $V$:} $g(A) = A$ for every $A \in \Omega$. Evidently, $g(\Omega) = \Omega$ and $g$ is valid (properly defined) on any $\Omega$. 
%
%\vspace{.5cm}
%
%
%\noi {\bf Theorem 4:}
\begin{thm}
Let $\Omega$ be a family of $v$ distinct non-empty subsets of $V$.
\be
\item[(a)]  Consider the following diagram:
$$ \Omega = {\Omega}_0 \xrightarrow{h_1} {\Omega}_1 \xrightarrow{h_2} {\Omega}_2 \xrightarrow{h_3} \cdots\cdots  \xrightarrow{h_n} {\Omega}_n = {\Omega}'$$ 
where each $h_i$ equals $g$ or $f_{A_i}$ and the functions are valid on the families they are defined. Then ${\Omega}' = g({\Omega}) = \Omega$ or ${\Omega}' = f_A(\Omega)$ for some $A \in \Omega$. 
\item[(b)] Let $\bf X$ be the set of all ${\Omega}'$ that can be obtained from $\Omega$ by a sequence of functions as in (a) and let $\bf Y$ be the set of all ${\Omega}''$ that can be obtained by a single function $h(\Omega) = {\Omega}''$ where $h = g$ or $h = f_A$ for some $A \in \Omega$. Then  $\bf X =  \bf Y$. 
\item[(c)] For two families $\Omega$ and $\Sigma$ of $v$ distinct nonempty subsets of $V$, write $\Sigma \sim \Omega$ if $\Sigma = h(\Omega)$ where $h = g$ or $h = f_A$ for some $A \in \Omega$. Then 
$\sim$ is an equivalent relation. 
\ee
\end{thm}
\begin{proof}
If some $h_i = g$ then we can effectively reduce the length of the sequence. Also, if ${\Omega}' = {\Omega}_n = \Omega$, then ${\Omega}' = g(\Omega)$ and we are done. The proof of (a) is by induction on $n$. The statement clearly holds for $n = 1$. Let 
$$ \Omega = {\Omega}_0 \xrightarrow{h_1} {\Omega}_1 \xrightarrow{h_2} {\Omega}_2 \xrightarrow{h_3} \cdots\cdots  \xrightarrow{h_n} {\Omega}_n = {\Omega}' \xrightarrow{h_{n + 1}} {\Omega}_{n + 1} = {\Omega}''$$
If ${\Omega}' = \Omega$, then ${\Omega}'' = g(\Omega)$ or ${\Omega}'' = f_A(\Omega)$ for some $A \in \Omega$ and we are done. Let ${\Omega}' \neq \Omega$. Then by the induction hypothesis, ${\Omega}' =  f_A(\Omega)$ for some $A \in \Omega$. If $h_{n + 1} = g$, then ${\Omega}'' = {\Omega}' = f_A(\Omega)$ and we are done. Otherwise, 
$h_{n + 1} = f_{A \bigtriangleup B}({\Omega}')$ for some $A \bigtriangleup B \in {\Omega}'$, that is, ${\Omega}'' = f_{A \bigtriangleup B}f_A(\Omega) = f_B(\Omega)$ for some $B \in \Omega$ (by Lemma \ref{lemma:map}) proving (a). Consider (b). Clearly ${\bf Y} \subset {\bf X}$. Using (a), if ${\Omega}' \in {\bf X}$, then either ${\Omega}' = \Omega$ or 
${\Omega}' = f_A(\Omega)$ for some $A \in \Omega$ showing that 
${\Omega}' \in {\bf Y}$. Hence, $\bf X =  \bf Y$. Note that reflexivity and symmetry of $\sim$ are taken care of by the function
$g$ and the fact that $f_A(f_A(\Omega)) = \Omega$ and (a) clearly proves transitivity. Thus $\sim$ is indeed an equivalence relation proving (c). 
\end{proof}

%\noi {\bf Corollary 5:} 
\begin{corollary}
Let $\Omega = \{A_1, A_2, \cdots, A_v\}$ be a family of $v$ distinct nonempty subsets of $V$. Let ${\Omega}_i = 
f_{A_i}(\Omega)$ and let $(\Omega)$ denote the equivalence class of $\Omega$. Then:
$$ (\Omega) = \{\Omega\} \cup \{{\Omega}_i: i = 1, 2, \cdots, v\}$$
\end{corollary}
%\vspace{.5cm}

%\noi {\bf Definition:} 
\begin{defn}
Let $\Omega = \{A_1, A_2, \cdots, A_v\}$ be a family of $v$ distinct nonempty subsets of $V$. For a point $x \in V$, let $r(x)$ denote the replication number (the number of blocks containing $x$) of $x$. Suppose we have constants $r_1 >r_2$ such that 
\be
\item[(i)] $r_1 + r_2 = v + 1$. 
\item[(ii)] $r(x) = r_1$ or $r(x) = r_2$ for every $x \in V$ (the possibility of all replication numbers being equal is also admissible).  
\ee
Then $\Omega$ is called a {\bf Ryser system with parameter triple $(v, r_1, r_2)$}. 
\end{defn}
%\vspace{.5cm}

%\noi {\bf Theorem 6:} 
\begin{thm}
Let $\Omega$ be a Ryser system with parameter triple $(v, r_1, r_2)$. Then for every ${\Omega}' \in (\Omega)$ the Ryser system ${\Omega}'$ also has the same parameter triple. 
\end{thm}
\begin{proof}
Let $f_A({\Omega}) = {\Omega}'$. If $x \notin A$, then $r_{{\Omega}'}(x) = r_{\Omega}(x)$ which equals $r_1$ or $r_2$ since $\Omega$ is a Ryser system. Let $x \in A$ and suppose w.l.o.g. that $r_{\Omega}(x) = r_1$. Then $f_A(\Omega)$ has exactly $(v - 1) - (r_1 - 1) = r_2 - 1$ (other than $A$ itself) that contain $x$. Hence 
$r_{{\Omega}'}(x) = (r_2 - 1) + 1 = r_2$. 
\end{proof}
%\vspace{.5cm}

%\noi {\bf Theorem 7(Ryser-Woodall):} 
\begin{thm}[{\cite[Proposition 14.1.7]{Ion2}}]\label{thm:equivalence-1}
Let $\mathcal{D}$ be a Ryser design on $v$ points. Then it is also a Ryser system with parameter triple 
$(v, r_1, r_2)$ and all Ryser designs in the equivalence class of 
$\mathcal{D}$ have the same parameter triple $(v, r_1, r_2)$. 
\end{thm}
%\vspace{.5cm}
%\noi {\bf Theorem 8:}
\begin{thm}\label{thm:equivalence-2}
Let $\mathcal{D}$ be a Ryser design on $v$ points. Then the following conditions are equivalent: 
\be
\item[(a)] The equivalence class of $\mathcal{D}$ (under $\sim$) contains a 
symmetric design $\mathcal{E}$. 
\item[(b)] $\mathcal{D}$ is of Type-1.  
\item[(c)] Every Ryser design in the equivalence class of $\mathcal{D}$ is 
of Type-1. 
\ee
\end{thm}
\begin{proof} Let ${\mathcal{D}}^*$ be a symmetric design in the equivalence class of $\mathcal{D}$ and let ${\mathcal{D}}'$ be some Ryser design in the equivalence class of $\mathcal{D}$. Then there is some block $A$ in 
${\mathcal{D}}^*$ such that block complementation with respect to that block produces ${\mathcal{D}}'$. This proves that (a) implies both (b) and (c). Clearly (c) is stronger than (b). Finally equivalence of the relation $\sim$ shows that (b) implies (a). 
\end{proof}
%\vspace{.5cm}

%\noi {\bf Theorem 9:} 
\begin{thm}
{\it Assume the following hypothesis: Every Ryser design ${\mathcal{D}}$ that has a block $C$ of even size is of Type-1.} Then the Ryser design conjecture is true. 
\end{thm}
\begin{proof}
Since the Ryser design conjecture holds for small 
values of $\lambda$, we may assume that the given Ryser design $\mathcal{D}$ is one with $\lambda \geq {\lambda}_0 \geq 2$. Since every block
must have size $\geq \lambda$ it follows (pigeonhole principle) that we have two blocks $A$ and $B$ of the same size $k$. Complementing $\mathcal{D}$ w.r.t. $A$ then renders $A \bigtriangleup B$ to have size $2(k - \lambda)$ which is even. Since this is a block in an equivalent design $\mathcal{E}$, the hypothesis implies $\mathcal{E}$ is of Type-1 and therefore by Theorem \ref{thm:equivalence-2} $\mathcal{D}$ is also of Type-1. 
\end{proof}

\section{\bf{The main results}}
%\section{\bf{Some observations in complementing a Type-2 Ryser design} }

\begin{thm}\label{thm:blocks_in_new_design}
	Let $\mathcal{D}$ be a Ryser design of Type-2 of order $v$ and index $\lambda$. Let $A$ be a non average block. Let $\overline{\mathcal{D}}={\mathcal{D}}*A$ be the new Ryser design of order $v$ and index $\overline{\lambda}=(|A|-\lambda)$ obtained from $\mathcal{D}$ by block complementation with respect to the block $A$. If $A$ is a large (respectively small) block in $\mathcal{D}$ then $A$ is a small (respectively large) block in $\overline{\mathcal{D}}$. Let $B$ be any other block of $\mathcal{D}$ and let $\overline{B}=A\triangle B$ be the new block in $\overline{\mathcal{D}}$ obtained from $B$. Then:
	\begin{enumerate}[(i)]
		\item If $|B|>|A|$ (respectively $|B|<|A|$) then $\overline{B}$ is a large (respectively small) block in $\overline{\mathcal{D}}$.
		\item If $|B|=2\lambda$ then $|\overline{B}|=|A|$ in $\overline{\mathcal{D}}$.
		\item If $|B|=|A|$ then $|\overline{B}|=2\overline{\lambda}$ in $\overline{\mathcal{D}}$.
	\end{enumerate}
\end{thm}
\begin{proof}
	Let $|A|=k=2\lambda+ta$ for some integer $t$. Then $A$ is a large block if $t>0$ and $A$ is a small block if $t<0$.
	If $k>2\lambda$ then $\overline{\lambda}>\lambda$ and $2\overline{\lambda}=2(k-\lambda)=k+(k-2\lambda)>k$. Hence if $A$ is a large block of $\mathcal{D}$ then it becomes a small block in $\overline{\mathcal{D}}$. The other case is similar.
	Let $B$ be any other block of $\mathcal{D}$ of size $ 2\lambda +t'a$. If $t'=0$ then, $B$ is an average block and if $t'\neq 0$ then, $B$ is a non average block.
	%After block complementation of $\mathcal{D}$ with respect to $A$ the block $B$ becomes $\overline{B}=A\triangle B$ 
	In the new design $\overline{\mathcal{D}}$ we have  $|\overline{B}|=|A\triangle B|=|A|+|B|-2|A\cap B|$.
	That is $|\overline{B}|=|A|+|B|-2\lambda=2(k-\lambda)+|B|-|A|=2\overline{\lambda}+|B|-|A|$.
	%This equation implies proof of all the three cases.
	This completes the proof.
\end{proof}	

\begin{lemma}\label{lemma:v-1avarage_blocks}
	Let $\mathcal{D}$ be a Ryser design of order $v$ and index $\lambda$. If $\mathcal{D}$ has $v-1$ blocks of size $2\lambda$ then $\mathcal{D}$ is of Type-1.
\end{lemma}
\begin{proof}
	Let $A$ be the only non-average block of $\mathcal{D}$ of size $k$. Consider $\overline{\mathcal{D}}=\mathcal{D}*A$. Then by Lemma \ref{lemma:blocksize} $\overline{\mathcal{D}}$ has all $v$ blocks of size $k$ with block intersection $k-\lambda$. Hence $\overline{\mathcal{D}}$ is a symmetric $(v,k,\lambda')$ design with $\lambda'=k-\lambda$. Therefore $\mathcal{D}$ is a Ryser design of Type-1.   	 
\end{proof}
%\section{\bf{Ryser design with two block sizes with average block as one block size}}
\begin{thm}\label{thm:MainResult}
	Let $\mathcal{D}$ be a Ryser design of order $v$ and index $\lambda$ with two block sizes such that one block size is $2\lambda$. Then $ \mathcal{D}$ is of Type-1.
\end{thm}
\begin{proof}
	Let the Ryser design $\mathcal{D}$ have two block sizes $k$ and $2\lambda$ where $k\neq2\lambda$.
	Let there be $\alpha$ blocks of size $k$ and $\beta$ blocks of size 
	$2\lambda$, where $\alpha,\beta\geq 1.$ Then we have
	\begin{equation}\label{alpha_plus_beta_is_v}
	\alpha+\beta=v
	\end{equation}
	Using (\ref{sum_of_all}) we get $ 
	\dfrac{\alpha}{k-\lambda}+\dfrac{\beta}{\lambda}=\dfrac{(\rho 
		+1)^2}{\rho}-\dfrac{1}{\lambda}$. Hence we have\\
	\begin{equation}\label{two_blocks_sum_of_all}
	\frac{\alpha}{k-\lambda}+\frac{\beta+1}{\lambda}=\frac{(\rho 
		+1)^2}{\rho}=\dfrac{(v-1)^2}{(r_1-1)(r_2-1)}=\dfrac{(v-1)^2}{r_1r_2-v}
	\end{equation}
	Since the total row sum of the incidence structure of a Ryser design is equal to the total column sum, we have $e_1r_1+e_2r_2=k\alpha+2\lambda\beta$. 
	From equation (\ref{e1r1e2r2form}) we know that 
	$e_1r_1+e_2r_2=\lambda(v-1)+r_1r_2$.\\
	Hence $k\alpha+2\lambda\beta=\lambda(v-1)+r_1r_2$ which obtains,
	\begin{equation}\label{r1r2_is_k-lambdaalpha_plus_lambdabeta+1}
	r_1r_2=(k-\lambda)\alpha+\lambda(\beta+1)
	\end{equation}
	Now using (\ref{r1r2_is_k-lambdaalpha_plus_lambdabeta+1}) in 
	(\ref{two_blocks_sum_of_all}) we have    
	$\dfrac{\alpha}{k-\lambda}+\dfrac{\beta+1}{\lambda}=\dfrac{(v-1)^2}{
		(k-\lambda)\alpha+\lambda(\beta+1)-v}$. \\
	Simplification then yields,\\
	
	$\alpha^2(k-2\lambda)^2-\alpha(k-2\lambda)[v+(k-2\lambda)(v+1)]
	+(k-\lambda)v(v+1-4\lambda)=0$.\\
	Let
	\begin{equation}\label{quadraticequation} 	
	P(\alpha)= \alpha^2(k-2\lambda)^2-\alpha(k-2\lambda)[v+(k-2\lambda)(v+1)]
	+(k-\lambda)v(v+1-4\lambda).
	\end{equation}
	We claim that $\alpha=v$ and $\alpha=1$ are the only roots to the 
	quadratic $P(\alpha)$.
	
	\noindent
	We have, 
	$P(v)=v^2(k-2\lambda)^2-v(k-2\lambda)[v+(k-2\lambda)(v+1)]
	+(k-\lambda)v(v+1-4\lambda)$. Simplification  then obtains, 
	$P(v)=v[-k(k-1)+\lambda(v-1)]$
	which is zero. %and hence $\alpha=v$ is a root of the quadratic 
	%	(\ref{quadraticequation}).
	If $\alpha=v$, then the Ryser design in consideration with two 
	block sizes becomes a symmetric design $(v,k,\lambda)$ and hence it satisfies the relation $k(k-1)=\lambda(v-1)$. \\ 
	Observer that,  
	$P(1)=(k-2\lambda)^2-(k-2\lambda)[v+(k-2\lambda)(v+1)]
	+(k-\lambda)v(v+1-4\lambda)$. Simplification  then obtains, 
	$P(1)=v[k(v-k)-\lambda(v-1)]$.
	By Lemma \ref{lemma:v-1avarage_blocks} we know that if $\alpha=1$, then the Ryser design in consideration with two 
	block sizes is of Type-1.
	Let this Ryser design be derived from a symmetric design 
	$(v,k,\lambda')$. Then $k(k-1)=\lambda'(v-1)$ and in the Ryser design 
	$\lambda=k-\lambda'$ therefore 
	$k^2-k=\lambda'v-\lambda'\Rightarrow  k^2=(k-\lambda)v+\lambda$ which 
	gives us $k(v-k)=\lambda(v-1)$.
	Therefore $\alpha=1$ is a root of the quadratic (\ref{quadraticequation}). 
	
	\noindent
	Successively differentiating quadratic (\ref{quadraticequation})
	obtains $P''(\alpha)=2(k-2\lambda)^2>0$ for all $ 1<\alpha<v$ because	$k\neq2\lambda$. Therefore $P(\alpha)$ does not change sign in the interval $(1,v)$ and 
	hence does not have any root in the interval $(1,v)$. 
	This proves that there is only one Ryser design with two block sizes in 
	which one of the block sizes is $2\lambda$.
\end{proof}
%\newpage
%\section{\bf{characterization of Type-1 Ryser designs}}
\noindent
%\begin{defn}{\rm
%		Two Ryser designs $\mathcal{D}_1=(X,L_1)$ and $\mathcal{D}_2=(X,L_2)$ with $|X|=v$ are said to be equivalent if one can be obtained from the other using Ryser-Woodall complementation procedure. 
%}\end{defn}
\noindent
By Theorem \ref{thm:equivalence-1}  all the Ryser designs in the same equivalence class have same value of $r_1$ and $r_2$ and hence have same value of $\rho$.
%To understand the reason to consider Ryser design with two block sizes with one of the block size as $2\lambda$ consider the following hypothesis 
Understanding the importance of the situation of two block
sizes with $2\lambda$ is helped by the following crucial hypothesis:\\
Hypothesis $\mathsf{H}$: Given any Ryser design $\mathcal{E}$, no design in the equivalence class  $(\mathcal{E})$ of  $\mathcal{E}$ have both a large and a small block, that is a large and a small block do not coexist in $\mathcal{E}$ or in any design equivalent to $\mathcal{E}$.
%\begin{thm}
%	Under the assumption that large and small blocks do not coexist in any 
%Ryser designs on $v$ points, every Type II Ryser design reduces to design with 
%two block sizes which has $\alpha$ number of large (small) blocks all of same 
%size and $v-\alpha$ average blocks where $v>\alpha\geq1$.
%\end{thm}
\begin{thm}
	Assuming the hypothesis $\mathsf{H}$ (to hold for all the Ryser designs),
	every Ryser design is of Type-1 and thus hypothesis $\mathsf{H}$ implies the validity
	of the Ryser design conjecture.
%	
%	Under the assumption that large and small blocks do not coexist in any equivalence class of a Ryser design of order $v$, every Type-2 Ryser design reduces to design with two block sizes with average block as one of the block size.
\end{thm}
\begin{proof}
	Let us assume that small and large blocks do not coexist in any equivalence class of a Ryser design.
	Consider a Ryser design $\mathcal{D}$ of order $v$ and index 
	$\lambda$ that does not have a small block. If $\mathcal{D}$ has at least two large blocks of different sizes say $B_1, B_2~~\text{ with } ~~|B_1|>|B_2|>2\lambda$ and if we complement this design with respect to block $B_2$. Then by Theorem \ref{thm:blocks_in_new_design} $~~$ we have a new design $\overline{\mathcal{D}}=\mathcal{D}*B_2$ in which $\overline{B_1}$ is a large block and $B_2$ is a small block, a contradiction. Therefore there can not exist two different block sizes of large blocks in the original design $\mathcal{D}$.
	Hence $\mathcal{D}$ is a Ryser design of order $v$ and index $\lambda$ with two block sizes with one block size $2\lambda$. Then by Theorem \ref{thm:MainResult} $ \mathcal{D}$ is of Type-1.
	% Hence under the assumption that large and small blocks do not coexist in any equivalence class of a Ryser design of order $v$ a Ryser design of Type-2 has average blocks and several large blocks all of the same size.\\
	%	As seen in the previous section complementing this Ryser design 
	%	of index $\lambda$ with respect to any large block gives us new 
	%	Ryser design of index $\overline{\lambda}$ with two block 
	%	sizes with one of the block size as $2\overline{\lambda}$ and the other of size $k<2\overline{\lambda}$. Hence we get a Type-2 design with two block sizes having average blocks and small blocks, all small blocks of the same size. 
\end{proof}
%\newpage
\begin{corollary}
	Ryser design conjecture is equivalent to any one of the following statements:
	\begin{enumerate}[(i)]
		\item Large and small blocks do not coexist in any Ryser design.
		\item There are no two different large (respectively small) block sizes in any Ryser design.
		\item There are exactly two block sizes in any Ryser design with one block size average.
	\end{enumerate}
\end{corollary}
%Hence to prove Ryser design conjecture it is sufficient to prove that large and small blocks do not coexist in any equivalence class of Ryser designs on $v$ points, or equivalently there are no two different large ( respectively small) block sizes in any equivalence class; or there are exactly two block sizes in every equivalence class with one of the block size as average.

\end{document}